\documentclass[12pt,reqno]{amsart}
\usepackage[T2A]{fontenc}
\usepackage[cp1251]{inputenc}

\usepackage{geometry} 
\usepackage{amsmath, amsthm, amssymb}
\usepackage{hyperref}
\usepackage{epigraph}
\usepackage{mathtools}
\usepackage{longtable}
\usepackage{pifont}
\usepackage{scalerel}
\usepackage{array}
\usepackage{lipsum}
\usepackage{scrextend}
\hypersetup{colorlinks=true,linkcolor= Mahogany, linktocpage, citecolor = OliveGreen}
\usepackage{graphicx} 
\usepackage{mathrsfs} 
\usepackage{color}
\usepackage{verbatim}
\usepackage[dvipsnames]{xcolor}
\usepackage{leftidx}

\usepackage[new]{old-arrows}

\usepackage{csquotes}

\usepackage{enumitem}
\setlist[enumerate,1]{label=(\arabic*),ref=\arabic*$^\circ$}

\usepackage{tikz-cd}
\usetikzlibrary{arrows}
\usepackage{graphicx}

\usepackage{geometry}

\usepackage[all]{xy}
\usepackage{hyperref}

\newcounter{eq}

\newcounter{eqintr}

\newtheorem{prop}[eq]{Proposition}
\newtheorem{thm}[eq]{Theorem}
\newtheorem{cor}[eq]{Corollary}
\newtheorem{lemma}[eq]{Lemma}
\theoremstyle{definition}
\newtheorem{df}[eq]{Definition}
\newtheorem*{example}{Example}
\newtheorem{notation}[eq]{Notation}
\theoremstyle{remark}
\newtheorem{rem}[eq]{Remark}

\numberwithin{eq}{section}

\DeclareMathOperator{\SheafHom}{\mathcal{H\kern -3pt o\kern -2pt m\kern -1pt}}

\newcommand{\A}{\mathbb A}

\newcommand{\Z}{\mathbb Z}

\address{\parbox{\linewidth}{
Department of Mathematics and Mechanics, Lomonosov Moscow State University, Moscow 119234, Russia
}}
\email{dmitrii.badulin@math.msu.ru, dbadulin28@gmail.com}

\let\oldtocsection=\tocsection

\let\oldtocsubsection=\tocsubsection

\let\oldtocsubsubsection=\tocsubsubsection

\renewcommand{\tocsection}[2]{\hspace{0em}\oldtocsection{#1}{#2}}
\renewcommand{\tocsubsection}[2]{\hspace{1em}\oldtocsubsection{#1}{#2}}
\renewcommand{\tocsubsubsection}[2]{\hspace{2em}\oldtocsubsubsection{#1}{#2}}

\begin{document}
\author{
Dmitry Badulin
}

\title{Structure of ind-pro completions of Noetherian rings
}
\date{}

\maketitle

\begin{abstract}
We prove some results on the structure of  
ind-pro completions of Noetherian rings along flags of prime ideals. In particular, we compute the Krull dimension and deduce the criterion on semilocality in the case of essentially of finite type algebras over a field. We also show that ind-pro completion inherits properties of the base ring such as normality, regularity, local equidimensionality, etc. 
\end{abstract}

\section*{Introduction}
In this paper, we examine properties of rings $C_\Delta R$, where $R$ is a Noetherian ring and $\Delta$ is a flag of prime ideals, i.e. ind-pro completions of Noetherian rings along flags of prime ideals, see section \ref{sect1} for definitions. Such rings appear as local factors of adelic groups, see \cite{Hu}, \cite{Be}, \cite{Pa}, \cite{Os}, \cite{Mo}. The rings $C_\Delta R$ are also known as {\it Beilinson completions}, see \cite{Ye}.

In \cite[Section 3.3]{Ye}, it was shown that $C_\Delta R$ is a finite product of $n$-dimensional local fields, where $R$ is an excellent integral domain and $\Delta = (\mathfrak{p}_0, \ldots, \mathfrak{p}_n)$ is a saturated flag of prime ideals of $R$, i.e. $\mathrm{ht}(\mathfrak{p}_i/\mathfrak{p}_{i - 1}) = 1$ for $1\leq i \leq n$, with $\mathfrak{p}_0 = (0)$. Moreover, if $\mathfrak{p}_0$ is arbitrary and $R$ is not necessarily an integral domain, where $\Delta$ is still saturated, then $C_\Delta R$ is a complete semilocal ring with Jackobson radical $\mathrm{rad}(C_\Delta R) = \mathfrak{p}_0C_\Delta R$, see \cite[Corollary 3.3.5]{Ye}. This allows us to compute the Krull dimension, namely $\dim C_\Delta R = \mathrm{ht}\:\mathfrak{p}_0 = \mathrm{ht}\:\mathfrak{p}_0 + \mathrm{ht}(\mathfrak{p}_n/\mathfrak{p}_0) - n$ (see proposition \ref{propDimCDelta}). The latter equality follows from the fact that $R$ is catenary and $\Delta$ is saturated.

In \cite{Ma3}, the fibers of the map $\mathrm{Spec}\: \hat{R}\rightarrow \mathrm{Spec}\: R$, which are called {\it formal fibers}, were examined, where $R$ is a Noetherian local ring and $\hat{R}$ is its completion. In particular, it was shown that if $R$ is an essentially of finite type algebra over a field, i.e. a localization of a finite type algebra over a field, of dimension $n$, then the supremum of dimensions of formal fibers is equal to $n - 1$. Moreover, this supremum is attained at some minimal prime ideal. If $R$ is an integral domain, then this result can be interpreted  as the dimension of the ind-pro completion for which the corresponding flag consists of the zero ideal and the maximal ideal.


We generalize the above results. Namely, we show that the dimension of the ind-pro completion $C_\Delta R$ with $\Delta = (\mathfrak{p}_0, \ldots, \mathfrak{p}_n)$ is equal to $\mathrm{ht}\:\mathfrak{p}_0 + \mathrm{ht}(\mathfrak{p}_n/\mathfrak{p}_0) - n$ in the case $R$ being an essentially of finite type algebra over a field. In this case, we also show that $C_\Delta R$ is semilocal if and only if $\Delta$ is saturated. We do this in two steps. First, we prove these results for the polynomial ring and the flag consisting of coordinate hyperplanes. The most intricate part of this step is to calculate the dimension of the generic formal fiber (see proposition \ref{propDimS-10}) to perform inductive argument. To do this, we explicitly construct prime ideals of the desired height, and in the case when $\Delta$ is not saturated, there are infinitely many such ideals. 
Our approach is based on the ideas from \cite{Ma3}. On the second step, we notice that we can consider $R$ as a finite type algebra over a field, and then we can deduce this case from the polynomial one by the Noether normalization lemma. 

We show that the ind-pro completion $C_\Delta R$ along a flag $\Delta$ is an excellent ring for any Noetherian ring $R$. We also show that $C_\Delta R$ is locally equidimensional if $R$ is universally catenary and locally equidimensional Noetherian ring. This follows from the fact that the notions of a formally catenary and universally catenary Noetherian ring are equivalent, see \cite{Ra}. 

The paper is organized as follows. In section \ref{sect1}, we examine general properties of ind-pro completions such as excellence, local equidimensionality, etc. Here we also show that the dimension of $C_\Delta R$ is equal to the sum of the dimension of the ring $R$ and the dimension of the fiber over $\mathfrak{p}_0$, see proposition \ref{propDimCDelta}. In section \ref{sect2}, we treat the case of the polynomial ring and the flag consisting of coordinate hyperplanes. In particular, we calculate the Krull dimension of the corresponding ind-pro completion and deduce the criterion of semilocality. In section \ref{sect3}, we treat the general case.

\subsection*{Acknowledgments}The author thanks his advisor Denis V. Osipov for his attention to this work. The author was supported by the Theoretical Physics and Mathematics Advancement Foundation ``BASIS'' under grant no. 25-8-2-9-1.

\section{General properties}\label{sect1}

Throughout the paper, we assume that all rings are commutative with identity. Recall some definitions and results from commutative algebra, see \cite{Ma1}, \cite{Ma2} for details.

\begin{df}\label{45}
    Let $R$ be a ring and let $\mathfrak{p}$ be a prime ideal. Let $M$ be an $R$-module. Then denote by $S^{-1}_\mathfrak{p}M$ or $M_\mathfrak{p}$ the localization of $M$ with respect to $\mathfrak{p}$ and denote by $C_\mathfrak{p}M$ the completion of $M$ with respect to $\mathfrak{p}$. Also denote $M^\wedge_\mathfrak{p} = C_\mathfrak{p}(M_\mathfrak{p})$.
\end{df}

\begin{df}
    Let $R$ be a ring. By a {\it flag of prime ideals} $\Delta$ we mean a non-empty set of prime ideals $\{\mathfrak{p}_0, \ldots, \mathfrak{p}_n\}$ of the ring $R$ for $n \geq 0$ such that $\mathfrak{p}_i\subsetneq \mathfrak{p}_{i + 1}$ for $0\leq i < n$. We will denote this by $\Delta = (\mathfrak{p}_0, \ldots, \mathfrak{p}_n)$. For a prime ideal $\mathfrak{p}\subset \mathfrak{p}_0$ denote by $\Delta\:\mathrm{mod}\:\mathfrak{p} = (\mathfrak{p}_0/\mathfrak{p}, \ldots, \mathfrak{p}_n/\mathfrak{p})$ the flag of prime ideals of the ring $R/\mathfrak{p}$.
\end{df}

\begin{df}
    Let $R$ be a Noetherian ring and let $\Delta = (\mathfrak{p}_0, \ldots, \mathfrak{p}_n)$ be a flag of prime ideals of $R$. Then denote
    $$C_\Delta R = C_{\mathfrak{p}_0}S^{-1}_{\mathfrak{p}_0}\ldots C_{\mathfrak{p}_n}S^{-1}_{\mathfrak{p}_n}R.$$
    The ring $C_\Delta R$ is called the {\it ind-pro completion along the flag $\Delta$}.
    
    For an 
    $R$-module $M$ we define $C_\Delta M$ in the same way. Also set $C_\varnothing M = M$.
\end{df}

Note that $C_\Delta M = C_\Delta R\otimes_R M$ for a finitely generated $R$-module $M$ by \cite[Proposition 3.2.1]{Hu}. 

\begin{rem}
Note the difference between $C_\mathfrak{p}M$ and $C_\mathfrak{(p)}M$ for an $R$-module $M$. Namely, $C_\mathfrak{p}M$ is a completion of $M$ with respect to $\mathfrak{p}$-adic topology, and $C_\mathfrak{(p)}M = M^\wedge_\mathfrak{p}$ is a completion of $M_\mathfrak{p}$ with respect to $\mathfrak{p}$-adic topology.
\end{rem}

\begin{df}
    Let $R\overset{\varphi}{\longrightarrow} S$ be a morphism of rings. For an ideal $\mathfrak{a}\subset S$ denote by $\mathfrak{a}\cap R$ the ideal $\varphi^{-1}(\mathfrak{a})$. The ideal $\mathfrak{a}\cap R$ is called a {\it contraction} of $\mathfrak{a}$ to $R$.

    For a flag $\Delta = (\mathfrak{p}_0, \ldots, \mathfrak{p}_n)$ of prime ideals of $S$ denote $\Delta \cap R = (\mathfrak{p}_0\cap R, \ldots, \mathfrak{p}_n\cap R)$ if $\mathfrak{p}_i\cap R\neq \mathfrak{p}_{i + 1}\cap R$ for $0\leq i < n$. If for a flag $\Gamma$ of prime ideals of $R$ we have $\Delta \cap R = \Gamma$, then denote this by $\Delta \mid \Gamma$.
\end{df}

\begin{df}
For a prime ideal $\mathfrak{p}$ in a ring $R$ define its {\it height} $\mathrm{ht}\:\mathfrak{p}$ as a supremum of lengths of chains of prime ideals terminating at $\mathfrak{p}$. In other words, 
$$\mathrm{ht}\:\mathfrak{p} = \sup\{n\mid \exists\: \mathfrak{p}_0\subsetneq \mathfrak{p}_1\subsetneq\ldots\subsetneq\mathfrak{p}_n = \mathfrak{p}\}.$$

For a pair of prime ideals $\mathfrak{p}, \mathfrak{q}$ of $R$ such that $\mathfrak{p}\subset \mathfrak{q}$ define $\mathrm{ht}(\mathfrak{q}/\mathfrak{p})$ as 
$$\mathrm{ht}(\mathfrak{q}/\mathfrak{p}) = \sup\{n\mid \exists\: \mathfrak{p} = \mathfrak{p}_0\subsetneq \mathfrak{p}_1\subsetneq\ldots\subsetneq\mathfrak{p}_n = \mathfrak{q}\}.$$
We can interpret $\mathrm{ht}(\mathfrak{q}/\mathfrak{p})$ as the height of the prime ideal $\mathfrak{q}/\mathfrak{p}$ in the ring $R/\mathfrak{p}$.
\end{df}

\begin{df}
    Let $R$ be a ring. A chain $\mathfrak{p}_1\subsetneq \mathfrak{p}_2\subsetneq \ldots \subsetneq \mathfrak{p}_n$ of prime ideals of $R$ is called {\it saturated}, if there are no prime ideals between $\mathfrak{p}_i$ and $\mathfrak{p}_{i + 1}$ for all $1\leq i\leq n - 1$.
    
    A flag $\Delta = (\mathfrak{p}_0, \ldots, \mathfrak{p}_n)$ is called saturated if the chain of prime ideals $\mathfrak{p}_0\subset \mathfrak{p}_1\subset \ldots \subset \mathfrak{p}_n$ is saturated.
\end{df}

\begin{df}
For a ring $R$ define its {\it Krull dimension} as
$$\dim R = \sup_{\mathfrak{p}}\mathrm{ht}\:\mathfrak{p},$$
where supremum is taken over all prime ideals of $R$.
\end{df}

\begin{df}
Let $R$ be a ring. Then the {\it Jackobson radical} $\mathrm{rad}(R)$ is the intersection of all maximal ideals of $R$.
\end{df}

\begin{df}
A Noetherian ring $R$ is called {\it equidimensional} if $\dim R = \dim R/\mathfrak{p}$ for any minimal prime ideal $\mathfrak{p}$ of $R$.

A Noetherian ring $R$ is called {\it locally equidimensional} if the ring $R_\mathfrak{p}$ is equidimensional for any prime ideal $\mathfrak{p}$ of $R$.
\end{df}

\begin{df}
    A ring $R$ is called {\it catenary}, if for any pair of prime ideals $\mathfrak{p}, \mathfrak{q}$ such that $\mathfrak{p}\subset \mathfrak{q}$ every two maximal strictly increasing chains of prime ideals between $\mathfrak{p}$ and $\mathfrak{q}$ have the same finite length.  

    A ring $R$ is called {\it universally catenary}, if any finitely generated algebra over $R$ is catenary.
\end{df}

\begin{df}
Let $\Bbbk$ be a field. Then a ring $R$ is an {\it essentially of finite type algebra over $\Bbbk$} if there exists a finite type $\Bbbk$-algebra $R'$ and a multiplicatively closed subset $T\subset R'$ such that $R\cong T^{-1}R'$.
\end{df}

For a definition of an excellent ring, see \cite[Section 34]{Ma2}. Examples of excellent rings are algebras of finite type over a field, over $\Z$, or over complete semilocal Noetherian rings.


The following proposition is a generalization of \cite[Corollary 2.1.2]{Ba}

\begin{prop}
   Let $R$ be a Noetherian ring and let $\Delta$, $\Gamma$ be flags of prime ideals of $R$ such that $\Delta\subset \Gamma$, or $\Delta = \varnothing$. Then the map of rings
   $$C_\Delta R\longrightarrow C_\Gamma R$$
   is flat.
\end{prop}
\begin{proof}
	We will proceed by induction on $|\Gamma|$. We can assume that $\Delta \neq \Gamma$, otherwise the statement is clear. If $\Delta = \varnothing$, then the assertion follows from \cite[Proposition 3.2.1]{Hu}. In particular, this holds for $|\Gamma| = 1$.
	
	Consider the case $|\Gamma| > 1$ and $\Delta \neq \varnothing$. Let $\Delta = (\mathfrak{p}, \ldots)$ and $\Gamma = (\mathfrak{q}, \ldots)$. First, let us assume $\mathfrak{p}\neq \mathfrak{q}$. Then $\Delta\subset \Gamma\setminus \mathfrak{q}$, and $C_\Gamma R$ is a completion of $S^{-1}_\mathfrak{q}C_{\Gamma\setminus \mathfrak{q}}R$. Since the latter ring is flat over $C_{\Gamma\setminus \mathfrak{q}}R$, the map $C_{\Gamma\setminus \mathfrak{q}}R \rightarrow C_\Gamma R$ is flat by \cite[Theorem 4.9]{GS}. From this we obtain that the composition 
	$$C_\Delta R \longrightarrow C_{\Gamma\setminus \mathfrak{q}} R\longrightarrow C_\Gamma R$$
 is flat by induction hypothesis.
		 
		 Now assume $\mathfrak{p} = \mathfrak{q}$. Since $C_\Delta R$ is complete with respect to $\mathfrak{p}C_\Delta R$-adic topology, $\mathfrak{p}C_\Delta R\subset \mathrm{rad}(C_\Delta R)$. Observe that 
	$$\mathfrak{p}C_\Delta R\otimes_{C_\Delta R}C_\Gamma R = \mathfrak{p}\otimes_R C_\Delta 	R\otimes_{C_\Delta R}C_\Gamma R =\mathfrak{p}\otimes_R C_\Gamma R = \mathfrak{p}C_\Gamma R,$$
since $C_\Delta R$ and $C_\Gamma R$ are flat $R$-modules by \cite[Proposition 3.2.1]{Hu}. By induction hypothesis, $C_{\Gamma\setminus \mathfrak{p}}R$ is flat over $C_{\Delta \setminus \mathfrak{p}}R$, from what follows that $C_{\Gamma\setminus \mathfrak{p}}R/\mathfrak{p}C_{\Gamma\setminus \mathfrak{p}}R$ is flat over $C_{\Delta\setminus \mathfrak{p}}R/\mathfrak{p}C_{\Delta\setminus \mathfrak{p}}R$, and the same holds for their localizations with respect to $\mathfrak{p}$. Observe that 
$$C_\Gamma R/\mathfrak{p}C_\Gamma R = S^{-1}_\mathfrak{p}C_{\Gamma\setminus \mathfrak{p}}R/\mathfrak{p}S^{-1}_\mathfrak{p}C_{\Gamma\setminus \mathfrak{p}}R, \quad C_\Delta R/\mathfrak{p}C_\Delta R = S^{-1}_\mathfrak{p}C_{\Delta\setminus \mathfrak{p}}R/\mathfrak{p}S^{-1}_\mathfrak{p}C_{\Delta\setminus \mathfrak{p}}R.$$
Therefore $C_\Gamma R/\mathfrak{p}C_\Gamma R$ is flat over $C_\Delta R/\mathfrak{p}C_\Delta R$. By \cite[Theorem 22.3]{Ma1}, $C_\Gamma R$ is flat over $C_\Delta R$.
\end{proof}

\begin{thm}\label{CDeltaExcellent}
    Let $R$ be a Noetherian ring and let $\Delta$ be a flag of prime ideals of $R$. Then the ring $C_\Delta R$ is excellent.
\end{thm}
\begin{proof}
    We will proceed by induction on $|\Delta|$. If $|\Delta| = 1$, then $C_\Delta R$ is a local complete Noetherian ring, which is an excellent ring (see \cite[Section 34]{Ma2}). Assume $|\Delta| > 1$. Let $\Delta = (\mathfrak{p}, \ldots)$. By induction hypothesis, $C_{\Delta\setminus \mathfrak{p}}R$ is an excellent ring. Then $S^{-1}_\mathfrak{p}C_{\Delta\setminus \mathfrak{p}}R$ is excellent as a localization of an excellent ring, and $C_\Delta R$ is excellent by \cite[Main Theorem 2]{KS}.
\end{proof}

In the following definition CM means Cohen--Macaulay. For details, see \cite{Ma2}.
\begin{df}
A Noetherian ring $R$ is called normal (resp. regular, resp. CM, resp. reduced), if $R_\mathfrak{p}$ is a normal local domain (resp. regular local ring, CM local ring, reduced local ring) for any $\mathfrak{p}\in \mathrm{Spec}\: R$.
\end{df}


The following proposition is a simple consequence of \cite[Theorem 79]{Ma2}

\begin{prop}\label{propRegNormCMRed}
Let $R$ be a normal (resp. regular, resp. CM, resp. reduced) excellent ring and let $\Delta$ be a flag of prime ideals of $R$, or $\Delta = \varnothing$. Then $C_\Delta R$ is a normal (resp. regular, resp. CM, resp. reduced) ring.
\end{prop}
\begin{proof}
We will proceed by induction on $|\Delta|$. The case $|\Delta| = 0$ follows from the assumptions. If $\Delta = (\mathfrak{p}, \ldots)$, then $C_\Delta R$ is a $\mathfrak{p}$-adic completion of the ring $S^{-1}_\mathfrak{p}C_{\Delta\setminus \mathfrak{p}}R$, and the latter ring is normal (resp. regular, resp. CM, resp. reduced) as a localization of a normal (resp. regular, resp. CM, resp. reduced) ring $C_{\Delta\setminus \mathfrak{p}} R$. Thus $C_\Delta R$ is normal (resp. regular, resp. CM, resp. reduced) by \cite[Theorem 79]{Ma2}.
\end{proof}

The next theorem is well-known in case of a saturated flag, see \cite[Section 3.3]{Ye}.
\begin{thm}\label{CDeltaRegularNormal}
Let $R$ be a Noetherian regular (resp. excellent normal) domain. Let $\Delta = (\mathfrak{p}_0, \ldots, \mathfrak{p}_n)$ be a flag of prime ideals of $R$ such that $R_{\mathfrak{p}_n}/\mathfrak{p}_iR_{\mathfrak{p}_n}$ is a regular (resp. normal) ring for $i = 0, 1, \ldots, n - 1$. Then $C_\Delta R$ is a regular (resp. normal) domain.
\end{thm}
\begin{rem}
 In other words, $C_\Delta R$ is a regular (resp. normal) domain whenever $\mathfrak{p}_n$ is a regular (resp. normal) point on $\mathrm{Spec}(R/\mathfrak{p}_i)$ for all $i = 0, \dots, n-1$.
\end{rem}
\begin{proof}
Since $C_\Delta R$ does not change if we consider $R_{\mathfrak{p}_n}$ instead of $R$, we can assume that $R/\mathfrak{p}_i$ is a regular (resp. normal) ring for $i = 0, 1, \ldots, n - 1$ without loss of generality.

We will proceed by induction on $n$.
First, let us treat the case $n = 0$. By our assumptions, $R$ is a Noetherian regular (resp. excellent normal) domain and $\mathfrak{p}_0$ is an arbitrary prime ideal.
If $R$ is a regular Noetherian ring, then $C_{(\mathfrak{p}_0)} R = R_{\mathfrak{p}_0}^\wedge$ is also regular Noetherian. It is also a local ring. In particular, it is a domain. If $R$ is an excellent normal domain, then $R_{\mathfrak{p}_0}$ is a local excellent normal domain, hence $C_{(\mathfrak{p}_0)} R= R_{\mathfrak{p}_0}^\wedge$ is a local normal ring by proposition \ref{propRegNormCMRed}, from what follows it is also a domain.

Let $n>0$. By our assumptions, $R$ is a Noetherian regular (resp. excellent normal) domain, $R/\mathfrak{p}_i$ is regular (resp. normal) for $0\leq i \leq n - 1$, and $\mathfrak{p}_n$ is an arbitrary prime ideal.
Denote $\Gamma = \Delta\setminus \mathfrak{p}_0$, $\tilde{\Gamma} = \Gamma \:\mathrm{mod}\:\mathfrak{p}_0$. By induction hypothesis, $C_\Gamma R$ and $C_{\tilde{\Gamma}}(R/\mathfrak{p}_0)$ are regular (resp. normal) domains.  It follows that $C_{\tilde{\Delta}}(R/\mathfrak{p}_0) = S^{-1}_{(0)}C_{\tilde{\Gamma}}(R/\mathfrak{p}_0)$ is a domain, since $\tilde{\Delta} = (0)\vee \tilde{\Gamma}$, where $\tilde{\Delta} = \Delta \:\mathrm{mod}\:\mathfrak{p}_0$. But $C_{\tilde{\Delta}}(R/\mathfrak{p}_0) = C_\Delta R/\mathfrak{p}_0C_\Delta R$, whence $\mathfrak{p}_0C_\Delta R$ is a prime ideal in $C_\Delta R$. 

By theorem \ref{CDeltaExcellent}, $C_\Gamma R$ is an excellent ring as well as its localization $S^{-1}_{\mathfrak{p}_0}C_\Gamma R$. Moreover, since $C_\Gamma R$ is a regular (resp. normal) ring, $S^{-1}_{\mathfrak{p}_0}C_\Gamma R$ is also regular (resp. normal). Therefore $C_\Delta R = C_{\mathfrak{p}_0}S^{-1}_{\mathfrak{p}_0}C_\Gamma R$ is a regular (resp. normal) Noetherian ring by \cite[Theorem 79]{Ma2}. In other words, $C_\Delta R$ is a finite product of regular (resp. normal) domains. But since $\mathfrak{p}_0C_\Delta R\subset \mathrm{rad}(C_\Delta R)$ and $\mathfrak{p}_0C_\Delta R$ is a prime ideal, we see that there is only one component in this product, i.e. $C_\Delta R$ is a regular (resp. normal) domain.
\end{proof}

The following theorem follows from Ratliff's results on the equivalence of universally catenary and formally catenary Noetherian local rings \cite{Ra}.

\begin{thm}\label{CDeltaLocallyEquidim}
    Let $R$ be a universally catenary locally equidimensional Noetherian ring and let $\Delta$ be a flag of prime ideals of $R$, or $\Delta = \varnothing$. Then the ring $C_\Delta R$ is locally equidimensional.
\end{thm}
\begin{proof}
	Let us proceed by induction on $|\Delta|$. If $|\Delta| = 0$, i.e. $\Delta = \varnothing$, then the assertion follows from the assumptions of the theorem. Assume that $|\Delta| > 0$. Let $\Delta = (\mathfrak{p}, \ldots)$. Then $C_{\Delta\setminus \mathfrak{p}}R$ is locally equidimensional by induction hypothesis. Let us denote $A = S^{-1}_\mathfrak{p}C_{\Delta\setminus \mathfrak{p}}R$ and $B = C_\Delta R$. Then $B$ is a $\mathfrak{p}A$-adic completion of $A$. 
	
	Observe that $A$ and $B$ are universally catenary by assumptions and by theorem \ref{CDeltaExcellent}. Note also that $A$ is locally equidimensional as a localization of a locally equidimensional ring. 
	
	By \cite[Corollary 2.19]{GS}, there is a bijection between the set of maximal ideals of $B$ and the set of those maximal ideals of $A$  which contain $\mathfrak{p}A$. Let $\mathfrak{M}$ be a maximal ideal of $B$ and let $\mathfrak{m} = \mathfrak{M}\cap A$ be a maximal ideal of $A$. Then there is a sequence of morphisms and the equality 
	$$A_\mathfrak{m}\longrightarrow B_\mathfrak{M}\longrightarrow B_\mathfrak{M}^\wedge = A_\mathfrak{m}^\wedge,$$
	where the second arrow is faithfully flat. 
	This follows from the fact that $A_\mathfrak{m}$ and $B_\mathfrak{M}$ are analytically isomorphic, i.e. they have the same completion, see \cite[Chapter 9, Section 24, (24.D)]{Ma2}. 
Since $A_\mathfrak{m}$ is universally catenary and equidimensional as a localization of $A$ at a maximal ideal, we obtain that $A_\mathfrak{m}^\wedge$ is equidimensional by \cite[Theorem 2.6]{Ra} (see also \cite[Tag 0AW6]{Sta}). Therefore $B_\mathfrak{M}^\wedge$ is equidimensional. It is also catenary as a complete local Noetherian ring. Then $B_\mathfrak{M}$ is equidimensional by \cite[Tag 0AW4]{Sta}. Note that the property of being a locally equidimensional ring for a catenary Noetherian ring is equivalent to the fact that the localizations at all maximal ideals of this ring are equidimensional. Therefore we obtain that $B$ is locally equidimensional.
\end{proof}

\begin{lemma}\label{lemmadimSdimR}
    Let $R$ be a local Noetherian ring with the maximal ideal $\mathfrak{m}$. Let $R\rightarrow S$ be a flat ring morphism such that $S$ is Noetherian and $\mathfrak{m}S\subset  \mathrm{rad}(S)$. Then
    $$\dim S = \dim R + \dim (S/\mathfrak{m}S).$$
\end{lemma}
\begin{proof}
	Note that $\dim R = \mathrm{ht}\:\mathfrak{m}$ is finite. 
	Let $\mathfrak{M}$ be a maximal ideal of $S$. Since $\mathfrak{m}S\subset \mathrm{rad}(S)$ and $R$ is local, $\mathfrak{M}\cap R = \mathfrak{m}$. By \cite[Theorem 15.1]{Ma1}, 
	$$\mathrm{ht}\: \mathfrak{M} = \dim R + \dim(S_\mathfrak{M}/\mathfrak{m}S_\mathfrak{M}).$$
	Hence
	$$\dim S = \sup_\mathfrak{M}\mathrm{ht}\:\mathfrak{M} = \dim R + \sup_\mathfrak{M}\dim(S_\mathfrak{M}/\mathfrak{m}S_\mathfrak{M}) = \dim R + \dim (S/\mathfrak{m}S).$$
	
\end{proof}

\begin{prop}\label{propDimCDelta}
    Let $R$ be a Noetherian ring and let $\Delta = (\mathfrak{p}_0, \ldots, \mathfrak{p}_n)$ be a flag of prime ideals. Then
    $$\dim C_\Delta R = \mathrm{ht}\: \mathfrak{p}_0 + \dim C_{\tilde{\Delta}}(R/\mathfrak{p}_0),$$
    where $\tilde{\Delta} = \Delta\:\mathrm{mod}\:\mathfrak{p}_0$.
\end{prop}
\begin{proof}
	Note that the morphism $R_{\mathfrak{p}_0}\rightarrow C_\Delta R$
	is flat and $\mathfrak{p}_0C_\Delta R\subset \mathrm{rad}(C_\Delta R)$, since the $C_\Delta R$ is $\mathfrak{p}_0C_\Delta R$-complete. Thus we obtain from lemma \ref{lemmadimSdimR}
	$$\dim C_\Delta R = \dim R_{\mathfrak{p}_0} + \dim C_\Delta R/\mathfrak{p}_0C_\Delta R.$$
	It remains to note that $\mathrm{ht}\:\mathfrak{p}_0 = \dim R_{\mathfrak{p}_0}$ and $C_\Delta R/\mathfrak{p}_0C_\Delta R = C_{\tilde{\Delta}}(R/\mathfrak{p}_0)$.
\end{proof}
\vspace{0.25cm}

\section{ Case $R = \Bbbk[x_1, \ldots, x_m]$}\label{sect2}
\begin{notation}\label{NotationPolynom}
    Let $R = \Bbbk[x_1, \ldots, x_m]$ and let $\mathfrak{q}_i = (x_1, \ldots, x_i)$ for $1\leq i \leq m$, $\mathfrak{q}_0 = (0)$. Let $0\leq k_0 < k_1 <\ldots < k_n\leq m$. Denote $\mathfrak{p}_j = \mathfrak{q}_{k_j}$ for $0\leq j \leq n$. Let $\Delta = (\mathfrak{p}_0, \ldots, \mathfrak{p}_n)$ be a flag of prime ideals of $R$.
\end{notation}
If $k_0 = 0$, then define the following notation
$$S[[x_1, \ldots, x_{k_0}]] := S$$
for a ring $S$.

Note that $R = \Bbbk[x_{k_0 + 1}, \ldots, x_m][x_1, \ldots, x_{k_0}] = (R/\mathfrak{p}_0)[x_1, \ldots, x_{k_0}]$, so there is a natural structure of an $R$-module on $C_{\tilde{\Delta}}(R/\mathfrak{p}_0)[[x_1, \ldots, x_{k_0}]]$,  where $\tilde{\Delta} = \Delta\:\mathrm{mod}\:\mathfrak{p}_0$.

\begin{lemma}\label{PolynomLocalFactor}
    In notation \ref{NotationPolynom}, there is an isomorphism of $R$-algebras
    $$C_\Delta R \cong C_{\tilde{\Delta}}(R/\mathfrak{p}_0)[[x_1, \ldots, x_{k_0}]],$$
    where $\tilde{\Delta} = \Delta\:\mathrm{mod}\:\mathfrak{p}_0$. In particular, $C_\Delta R$ is an integral domain.
\end{lemma}
\begin{proof}
	If $k_0 = 0$, then the assertion is clear. Consider the case $k_0 > 0$. We will proceed by induction on $|\Delta|$. Let $\Gamma = \Delta\setminus \mathfrak{p}_0$. Then 
	\begin{equation}\label{eqCDeltaprojlim}
	C_\Delta R = \varprojlim_{l}\: S^{-1}_{\mathfrak{p}_0}C_\Gamma R/\mathfrak{p}_0^lS^{-1}_{\mathfrak{p}_0}C_\Gamma R = \varprojlim_{l}\: S^{-1}_{\mathfrak{p}_0}(C_\Gamma R/\mathfrak{p}_0^lC_\Gamma R).
	\end{equation}
	If $\Gamma = \varnothing$, then $C_\Gamma R = R$, and thus $C_\Delta R = \Bbbk(x_{k_0 + 1}, \ldots, x_m)[[x_1, \ldots, x_{k_0}]]$, where $\Bbbk(x_{k_0 + 1}, \ldots, x_m) = \mathrm{Frac}(\Bbbk[x_{k_0 + 1}, \ldots, x_m])$, where we assume that $\Bbbk(x_{k_0 + 1}, \ldots, x_m) = \Bbbk$ if $k_0 = m$ and $n = 0$. Note that $C_{\tilde{\Delta}}(R/\mathfrak{p}_0) \cong  \Bbbk(x_{k_0 + 1}, \ldots, x_m)$ in this case. 
	
	Assume $\Gamma\neq \varnothing$, i.e. $n > 0$. By induction hypothesis, 
	$$C_\Gamma R \cong C_{\widehat{\Gamma}}(R/\mathfrak{p}_1)[[x_1, \ldots, x_{k_1}]],$$
	where $\widehat{\Gamma} = \Gamma\:\mathrm{mod}\: \mathfrak{p}_1$. Applying the induction hypothesis for the ring $R/\mathfrak{p}_0 = \Bbbk[x_{k_0 + 1}, \ldots, x_m]$ and the flag $\tilde{\Gamma} = \Gamma\:\mathrm{mod}\: \mathfrak{p}_0$, we obtain an isomorphism of $R/\mathfrak{p}_0$-algebras
	$$C_{\tilde{\Gamma}}(R/\mathfrak{p}_0)\cong C_{\widehat{\Gamma}}(R/\mathfrak{p}_1)[[x_{k_0 + 1}, \ldots, x_{k_1}]].$$
	Thus
	$$C_\Gamma R \cong C_{\widehat{\Gamma}}(R/\mathfrak{p}_1)[[x_1, \ldots, x_{k_1}]] = C_{\widehat{\Gamma}}(R/\mathfrak{p}_1)[[x_{k_0 + 1}, \ldots, x_{k_1}]][[x_1, \ldots, x_{k_0}]]\cong C_{\tilde{\Gamma}}(R/\mathfrak{p}_0)[[x_1, \ldots, x_{k_0}]]$$
	is an isomorphism of $R$-algebras. 
	
	Denote $B = C_{\tilde{\Gamma}}(R/\mathfrak{p}_0)$ and $B_0 = B_{(0)}$, where the latter ring is a localization of $B$ at the zero ideal of the ring $R/\mathfrak{p}_0$. We claim that there is an isomorphism of $R$-algebras
	$$S^{-1}_{\mathfrak{p}_0}(B[[x_1, \ldots, x_{k_0}]]/\mathfrak{p}_0^lB[[x_1, \ldots, x_{k_0}]]) \cong B_0[[x_1, \ldots, x_{k_0}]]/\mathfrak{p}_0^lB_0[[x_1, \ldots, x_{k_0}]]$$
	for $l\geq 1$.
	
	Let $f\in R\setminus \mathfrak{p}_0$. This means that there exists elements $g\in \Bbbk[x_{k_0 + 1}, \ldots, x_m]$ and $h\in \mathfrak{p}_0$ such that $f = g + h$. Since $g$ is invertible in $B_0$, $f$ is invertible in $B_0[[x_1, \ldots, x_{k_0}]]$. Therefore there exists an inclusion of $R$-algebras
	$$S^{-1}_{\mathfrak{p}_0}(B[[x_1, \ldots, x_{k_0}]])\longhookrightarrow B_0[[x_1, \ldots, x_{k_0}]].$$
	Since every non-zero element of $R/\mathfrak{p}_0 = \Bbbk[x_{k_0 + 1}, \ldots, x_m]$ can be naturally considered as an element of $R\setminus \mathfrak{p}_0$, there is an inclusion of $R$-algebras
	$$B_0[x_1, \ldots, x_{k_0}]\longhookrightarrow S^{-1}_{\mathfrak{p}_0}(B[[x_1, \ldots, x_{k_0}]]).$$
	Since $B\hookrightarrow B_0$ and $\mathfrak{p}_0B_0[[x_1, \ldots, x_{k_0}]] = (x_1, \ldots, x_{k_0})$, one can see 	that 
	\begin{align*}
	S^{-1}_{\mathfrak{p}_0}(B[[x_1, \ldots, x_{k_0}]])\cap \mathfrak{p}^l_0B_0[[x_1, \ldots, x_{k_0}]] &= S^{-1}_{\mathfrak{p}_0}(B[[x_1, \ldots, x_{k_0}]]\cap \mathfrak{p}^l_0B_0[[x_1, \ldots, x_{k_0}]]) \\
	&= S^{-1}_{\mathfrak{p}_0}(\mathfrak{p}^l_0B[[x_1, \ldots, x_{k_0}]])
	\end{align*}
	for $l\geq 1$. Since 
	$$B_0[x_1, \ldots, x_{k_0}]\cap \mathfrak{p}^l_0B_0[[x_1, \ldots, x_{k_0}]] = \mathfrak{p}^l_0B_0[x_1, \ldots, x_{k_0}],$$
	there is a sequence of inclusions
	\begin{align*}
	B_0[x_1, \ldots, x_{k_0}]/\mathfrak{p}^l_0B_0[x_1, \ldots, x_{k_0}]&\longhookrightarrow S^{-1}_{\mathfrak{p}_0}(B[[x_1, \ldots, x_{k_0}]]/\mathfrak{p}^l_0B[[x_1, \ldots, x_{k_0}]])\\
	&\longhookrightarrow B_0[[x_1, \ldots, x_{k_0}]]/\mathfrak{p}^l_0B_0[[x_1, \ldots, x_{k_0}]],
	\end{align*}
	and since the composition is an isomorphism, we obtain the required isomorphism.
	
	Therefore we obtain
	\begin{align*}
	C_\Delta R &= \varprojlim_{l}\: S^{-1}_{\mathfrak{p}_0}(C_\Gamma R/\mathfrak{p}_0^lC_\Gamma R) = \varprojlim_l \: S^{-1}_{\mathfrak{p}_0}(B[[x_1, \ldots, x_{k_0}]]/\mathfrak{p}^l_0B[[x_1, \ldots, x_{k_0}]]) \\
	&= \varprojlim_l\: B_0[[x_1, \ldots, x_{k_0}]]/\mathfrak{p}^l_0B_0[[x_1, \ldots, x_{k_0}]] = B_0[[x_1, \ldots, x_{k_0}]], \tag{\theeq}\label{eqCDeltaSeries}
	\end{align*}
	and it remains to note that $B_0 = S^{-1}_{(0)}B = S^{-1}_{(0)}C_{\tilde{\Gamma}}(R/\mathfrak{p}_0) = C_{\tilde{\Delta}}(R/\mathfrak{p}_0)$.
	
	The fact that $C_\Delta R$ is an integral domain follows from theorem \ref{CDeltaRegularNormal}, since $R$ and each $R/\mathfrak{p}_i$ are polynomial algebras over a field. One can also deduce this directly from \eqref{eqCDeltaSeries} by induction on the length of the flag.	
\end{proof}

The following proposition is a generalization of  \cite[Theorem 2]{Ma3}.

\begin{prop}\label{propDimS-10}
    In notation \ref{NotationPolynom}, if $k_0 > 0$, there is an equality
    $$\dim S^{-1}_{(0)}C_\Delta R = \dim C_\Delta R - 1.$$
\end{prop}
\begin{proof}
	By lemma \ref{PolynomLocalFactor},
	$$C_\Delta R \cong C_{\tilde{\Delta}}(R/\mathfrak{p}_0)[[x_1,\ldots, x_{k_0}]].$$
	Let $\mathfrak{M}$ be a maximal ideal of $C_{\tilde{\Delta}}(R/\mathfrak{p}_0)$ of height $\dim C_{\tilde{\Delta}}(R/\mathfrak{p}_0)$. Let $K = C_{\tilde{\Delta}}(R/\mathfrak{p}_0)/\mathfrak{M}$. Since $\tilde{\Delta} = (0, \ldots)$, the ideal $\mathfrak{M}$ intersects with $R/\mathfrak{p}_0 = \Bbbk[x_{k_0 + 1}, \ldots, x_m]$ by the zero ideal, and thus $R/\mathfrak{p}_0\hookrightarrow K$. 
	
	Let $u_1 = y, u_2, \ldots, u_{k_0}$ be algebraically independent elements of $K[[y]]$ over $K$ with zero constant terms. Such elements exist, since $\mathrm{tr.deg}_KK[[y]] = \infty$. Consider the composition map
	$$C_\Delta R\cong C_{\tilde{\Delta}}(R/\mathfrak{p}_0)[[x_1,\ldots, x_{k_0}]] \longrightarrow K[[x_1,\ldots, x_{k_0}]] \longrightarrow K[[y]],$$
	where the first map is a factorization of the coefficient ring by the ideal $\mathfrak{M}$ and the second map is given by $x_i\mapsto u_i$ for $1\leq i \leq k_0$. By construction, this map is surjective. Let $\mathfrak{P}$ be the kernel of this map.
	
	Let $f\in R\cap \mathfrak{P}$. Since there is a natural isomorphism of rings $R\cong (R/\mathfrak{p}_0)[x_1, \ldots, x_{k_0}]$, we can consider $f$ as a polynomial $F(x_1, \ldots, x_{k_0}) \in (R/\mathfrak{p}_0)[x_1, \ldots, x_{k_0}]$. Thus $\bar{F}(u_1, \ldots, u_{k_0}) = 0$ in $K[[y]]$, where $\bar{F} = F\:\mathrm{mod}\: \mathfrak{M}$. Since $u_1, \ldots, u_{k_0}$ are algebraically independent over $K$, we obtain that coefficients of $\bar{F}$ are zero in $K$. Since $R/\mathfrak{p}_0\hookrightarrow K$, we obtain that $F = 0$ in $(R/\mathfrak{p}_0)[x_1, \ldots, x_{k_0}]$, and thus $f = 0$. This means that $R\cap \mathfrak{P} = (0)$, and therefore $\mathfrak{P}S^{-1}_{(0)}C_\Delta R$ is a prime ideal of $S^{-1}_{(0)}C_\Delta R$.
	
	Observe that
	$$\mathrm{ht}\ker(K[[x_1,\ldots, x_{k_0}]] \rightarrow K[[y]]) = k_0 - 1,$$
	since the ring $K[[x_1, \ldots, x_{k_0}]]$ is local and the map is surjective. Note that 
	$$\ker(K[[x_1,\ldots, x_{k_0}]] \rightarrow K[[y]]) = \mathfrak{P}/\mathfrak{M}[[x_1, \ldots, x_{k_0}]].$$
	Note also that the height of the ideal $\mathfrak{M}[[x_1, \ldots, x_{k_0}]]$ in the ring $C_{\tilde{\Delta}}(R/\mathfrak{p}_0)[[x_1,\ldots, x_{k_0}]]$ equals the height of the ideal $\mathfrak{M}$ in the ring $C_{\tilde{\Delta}}(R/\mathfrak{p}_0)$. Then, since $C_\Delta R$ is catenary and locally equidimensional by theorem \ref{CDeltaExcellent} and theorem \ref{CDeltaLocallyEquidim}, we have
	$$\mathrm{ht}\:\mathfrak{P} = \mathrm{ht}\:\mathfrak{M} + \mathrm{ht}(\mathfrak{P}/\mathfrak{M}[[x_1, \ldots, x_{k_0}]]) = \dim C_{\tilde{\Delta}}(R/\mathfrak{p}_0) + k_0 - 1 = \dim C_\Delta R - 1,$$
	where the last equality follows from proposition \ref{propDimCDelta}, since $\mathrm{ht}\:\mathfrak{p}_0 = k_0$. Since $\mathfrak{p}_0C_\Delta R\subset \mathrm{rad}(C_\Delta R)$ and $\mathfrak{p}_0\neq (0)$, we obtain that $\dim C_\Delta R\geq \dim S^{-1}_{(0)}C_\Delta R + 1$. But $\mathfrak{P}\cap R = (0)$ and $\mathrm{ht}\:\mathfrak{P} = \dim C_\Delta R - 1$, whence we obtain the desired equality.
\end{proof}

\begin{prop}\label{DimCDeltaPolynom}
    In notation \ref{NotationPolynom}, there is an equality
    $$\dim C_\Delta R = \mathrm{ht}\:\mathfrak{p}_n - n.$$
\end{prop}
\begin{proof}
	We will proceed by induction on $n$. If $n = 0$, then $C_\Delta R$ is a local ring of dimension $\mathrm{ht}\: \mathfrak{p}_0$. Thus we can assume $n > 0$. 
	
	By proposition \ref{propDimCDelta},
	$$\dim C_\Delta R = \mathrm{ht}\:\mathfrak{p}_0 + \dim C_{\tilde{\Delta}}(R/\mathfrak{p}_0).$$
	Denote $\Gamma = \Delta\setminus \mathfrak{p}_0$. Then $\tilde{\Delta} = (0, \ldots)$ and $\tilde{\Delta}\setminus 0 = \tilde{\Gamma}$. By proposition \ref{propDimS-10}, $\dim C_{\tilde{\Delta}}(R/\mathfrak{p}_0) = \dim C_{\tilde{\Gamma}}(R/\mathfrak{p}_0) - 1$. By induction hypothesis, $\dim C_{\tilde{\Gamma}}(R/\mathfrak{p}_0) = \mathrm{ht}(\mathfrak{p}_n/\mathfrak{p}_0) - (n - 1) = \mathrm{ht}\:\mathfrak{p}_n - \mathrm{ht}\:\mathfrak{p}_0 - n + 1$, where the last equality holds since $R$ is a polynomial ring over field. Therefore
	$$\dim C_\Delta R = \mathrm{ht}\:\mathfrak{p}_0 + \dim C_{\tilde{\Delta}}(R/\mathfrak{p}_0) = \mathrm{ht}\:\mathfrak{p}_0 + \mathrm{ht}\:\mathfrak{p}_n - \mathrm{ht}\:\mathfrak{p}_0 - n + 1 - 1 = \mathrm{ht}\:\mathfrak{p}_n - n.$$
\end{proof}

\begin{prop}\label{propSemilocalPolynom}
     In notation \ref{NotationPolynom}, $C_\Delta R$ is semilocal if and only if $\Delta$ is saturated.
\end{prop}
\begin{proof}
If $\Delta$ is saturated, then the semilocality follows from \cite[Corollary 3.3.5]{Ye}. Thus we can assume that $\Delta$ is not saturated. In particular, $|\Delta| > 1$.

Since $\mathfrak{p}_0C_\Delta R\subset \mathrm{rad}(C_\Delta R)$, the semilocality of the ring $C_\Delta R$ is equivalent to the semilocality of the ring $C_\Delta R/\mathfrak{p}_0C_\Delta R = C_{\tilde{\Delta}}(R/\mathfrak{p}_0)$, where $\tilde{\Delta} = \Delta\:\mathrm{mod}\:\mathfrak{p}_0$. Thus we can reduce the assertion to the ring $C_{\tilde{\Delta}}(R/\mathfrak{p}_0)$ or, in other words, we can assume that $R$ is an integral domain and $\mathfrak{p}_0 = (0)$. 

Denote $\Gamma = \Delta\setminus \mathfrak{p}_0$. Then $C_\Delta R = S^{-1}_{(0)}C_\Gamma R$, where
$$C_\Gamma R = C_{\widehat{\Gamma}}(R/\mathfrak{p}_1)[[x_1, \ldots, x_{k_1}]],$$
where $\widehat{\Gamma} = \Gamma\:\mathrm{mod}\:\mathfrak{p}_1$. If $k_1 = 1$ or, in other words, $\mathrm{ht}(\mathfrak{p}_1/\mathfrak{p}_0) = \mathrm{ht}\:\mathfrak{p}_1 = 1$, then $S^{-1}_{(0)}C_\Gamma R = C_{\widehat{\Gamma}}(R/\mathfrak{p}_1)((x_1))$ and $\widehat{\Gamma}$ is not saturated. 
Since different maximal ideals of $C_{\widehat{\Gamma}}(R/\mathfrak{p}_1)$ generate different maximal ideals of $C_{\widehat{\Gamma}}(R/\mathfrak{p}_1)((x_1))$, we can reduce the problem to the case of the ring $R/\mathfrak{p}_1$ and the flag $\widehat{\Gamma}$. We can assume by induction on the length of the flag that the ring $C_{\widehat{\Gamma}}(R/\mathfrak{p}_1)$ is not semilocal, and then so is $C_\Delta R$. 

Consider the remaining case $k_1 > 1$. Let us perform the same construction as in the proof of proposition \ref{propDimS-10}.
Let $\mathfrak{M}$ be a maximal ideal of $C_{\widehat{\Gamma}}(R/\mathfrak{p}_1)$ of height $\dim C_{\widehat{\Gamma}}(R/\mathfrak{p}_1)$. 
Denote $K = C_{\widehat{\Gamma}}(R/\mathfrak{p}_1)/\mathfrak{M}$. Since $\mathrm{tr.deg}_KK[[y]] = \infty$, there exists an infinite sequence of elements $u_1 = y, u_2, \ldots, u_{k_1 - 1}, v_1, v_2, \ldots$ of $K[[y]]$ with zero constant terms which are algebraically independent over $K$. For a positive integer $r$ consider the composition
$$C_\Gamma R = C_{\widehat{\Gamma}}(R/\mathfrak{p}_1)[[x_1, \ldots, x_{k_1}]]\longrightarrow K[[x_1, \ldots, x_{k_1}]] \longrightarrow K[[y, z]]\longrightarrow K[[y]],$$
where the first map is a factorization of the coefficient ring by the ideal $\mathfrak{M}$, the second map is given by $x_j\mapsto u_j$ for $1\leq j < k_1$, $x_{k_1}\mapsto z$, and the third map is given by $y\mapsto y$, $z\mapsto v_r$. Denote by $\mathfrak{r}_r$ the kernel of the third map and denote by $\mathfrak{P}_r$ the kernel of the composition. Since all maps in the above sequence are surjective, we see that $\mathfrak{P}_r = \mathfrak{r}_r\cap C_\Gamma R$. Note also that $\mathfrak{P}_r$ and $\mathfrak{r}_r$ are prime ideals, since $K[[y]]$ is an integral domain. By the same reasoning as in the proof of proposition \ref{propDimS-10}, we see that $\mathfrak{P}_r\cap R = (0)$ and $\mathrm{ht}\:\mathfrak{P}_r = \dim C_\Gamma R - 1$. In other words, $\mathfrak{P}_rC_\Delta R$ is the maximal ideal in $C_\Delta R = S^{-1}_{(0)}C_\Gamma R$.

We have constructed the set $\{\mathfrak{P}_rC_\Delta R\}_{r = 1}^\infty$ of maximal ideals of the ring $C_\Delta R$. To show that $C_\Delta R$ is not semilocal, it is sufficient to show that $\mathfrak{P}_i\neq \mathfrak{P}_j$ in $C_\Gamma R$ for $i\neq j$. Since $\mathfrak{P}_r = \mathfrak{r}_r\cap C_\Gamma R$ for any $r\geq 1$, it is sufficient to prove that $\mathfrak{r}_i\neq \mathfrak{r}_j$ in $K[[y, z]]$ for $i\neq j$.

First, note that $\mathfrak{r}_r = (z - v_r)$ for any $r\geq 1$. If $\mathfrak{r}_i = \mathfrak{r}_j$, then there exists $g\in K[[y, z]]$ such that  $z - v_i = g(z - v_j)$. If $g = \sum_{s = 0}^\infty c_sz^s$ with $c_s\in K[[y]]$, then 
$$z - v_i = \sum_{s = 1}^\infty (c_{s - 1} - c_sv_j)z^s - c_0v_j.$$
We can see that $c_{s - 1} - c_sv_j = 0$ for $s\geq 2$. This means that for any $s\geq 2$ we have $c_{s - 1} = c_sv_j = c_{s + 1}v_j^2 = c_{s + 2}v_j^3 = \ldots$, or 
$$c_{s - 1}\in \bigcap_{d\geq 1}(v_j)^d,$$
where $(v_j)\subset K[[y]]$ is a proper ideal, since $v_j$ has no constant term. Since $K[[y]]$ is a Noetherian integral domain, we obtain that $c_{s - 1} = 0$ for $s\geq 2$ by \cite[Theorem 8.10]{Ma1}. Thus we obtain that
$$z - v_i = c_0z - c_0v_j,$$
from what follows $c_0 = 1$ and $v_i = v_j$. Since $i\neq j$, we obtain a contradiction. Hence $\mathfrak{r}_i \neq \mathfrak{r}_j$
and $C_\Delta R$ is not semilocal.
\end{proof}
\vspace{0.25cm}

\section{General case}\label{sect3}
\begin{notation}\label{NotationGeneral}
Let $R$ be an essentially of finite type algebra over a filed $\Bbbk$ and let $\Delta = (\mathfrak{p}_0, \ldots, \mathfrak{p}_n)$ be a flag of prime ideals of $R$.
\end{notation}

\begin{rem}\label{remFinGenAlg}
	In the sequel, we will take completions of $R$ along the flag $\Delta$, in particular we will take localizations by the prime ideal $\mathfrak{p}_n$ on the first step. If $R = T^{-1}R'$, where $R'$ is a finitely generated $\Bbbk$-algebra and $T$ is a multiplicatively closed subset of $R'$, then localizations of $R$ and $R'$ at prime ideals $\mathfrak{p}_n$ and $\mathfrak{p}_n\cap R'$ coincide. Besides, heights of prime ideals of $R'$ which do not intersect $T$ stay the same after localization. This means that we can consider $R'$ instead of $R$ in the sequel or, in other words, we can assume that $R$ is finitely generated over $\Bbbk$.
\end{rem}

\begin{thm}\label{thmDimCDeltaR}
    In notation \ref{NotationGeneral}, we have
    $$\dim C_\Delta R = \mathrm{ht}\: \mathfrak{p}_0 + \mathrm{ht}(\mathfrak{p}_n/\mathfrak{p}_0) - n.$$
\end{thm}
\begin{proof}
	By remark \ref{remFinGenAlg}, assume that $R$ is a finitely generated $\Bbbk$-algebra. By proposition \ref{propDimCDelta},
	$$\dim C_\Delta R = \mathrm{ht}\:\mathfrak{p}_0 + \dim C_{\tilde{\Delta}}(R/\mathfrak{p}_0).$$
	Thus it is enough to prove that $\dim C_{\tilde{\Delta}}(R/\mathfrak{p}_0) = \mathrm{ht}(\mathfrak{p}_n/\mathfrak{p}_0) - n$. Hence we can assume that $R$ is an integral domain and $\mathfrak{p}_0 = (0)$.
	
	By \cite[Chapter V, \S 3.1, Theorem 1]{Bo}, there exists a polynomial subring $A = \Bbbk[x_0, \ldots, x_m]$ in $R$ such that $R$ is finite over $A$ and $\mathfrak{r}_i = \mathfrak{p}_i\cap A = (x_1, \ldots, x_{k_i})$ for some $0 = k_0  < k_1 <\ldots < k_n$, where we assume that $(0) = (x_1, \ldots, x_{k_0})$ for $k_0 = 0$. By \cite[Theorem 9.4]{Ma1}, $\mathrm{ht}\:\mathfrak{p}_i = \mathrm{ht}\:\mathfrak{r}_i$. Denote $\Gamma = (\mathfrak{r}_0, \ldots, \mathfrak{r}_n)$.
	Note that $\mathrm{Frac}(A)\otimes_A R$ is flat over the field $\mathrm{Frac}(A)$. Thus $(\mathrm{Frac}(A)\otimes_A R)\otimes_A C_\Gamma A$ is flat over $\mathrm{Frac}(A)\otimes_A C_\Gamma A$. 
	Since $\mathfrak{p}_0 = (0)$ and $\mathfrak{r}_0 = (0)$, we have
	$$\mathrm{Frac}(A)\otimes_A C_\Gamma A = C_\Gamma A, \quad (\mathrm{Frac}(A)\otimes_A R)\otimes_A C_\Gamma A = R\otimes_A C_\Gamma A = \prod_{\Delta'\mid \Gamma} C_{\Delta'}R,$$
	where the last equality follows from \cite[Proposition 3.1.7]{Ye}. Note that the product is finite, since $R$ is finite over $A$. By \cite[Theorem 9.5]{Ma1}, the going-down theorem holds between $C_\Gamma A$ and $\prod_{\Delta'\mid \Gamma} C_{\Delta'}R$. But $R$ is finite over $A$ as well as $R\otimes_A C_\Gamma A= \prod_{\Delta'\mid \Gamma} C_{\Delta'}R$ is finite over $C_\Gamma A$, so the going-up theorem holds between $\prod_{\Delta'\mid \Gamma} C_{\Delta'}R$ and $C_\Gamma A$ by \cite[Theorem 9.4]{Ma1}.
	
	Note that any chain of prime ideals $\prod_{\Delta'\mid \Gamma} C_{\Delta'}R$ is a chain of prime ideals in $C_{\Delta'}R$ for some $\Delta'\mid \Gamma$. From this and from the fact that the going-down and the going-up theorems hold for the ring morphism $C_\Gamma A\rightarrow \prod_{\Delta'\mid \Gamma} C_{\Delta'}R$, it follows that the going-down and the going-up theorems hold for the ring morphism $C_\Gamma A\rightarrow C_{\Delta'}R$ for any $\Delta'\mid \Gamma$. Thus $\dim C_\Delta' R = \dim C_\Gamma A$ for any $\Delta'\mid \Gamma$, since $C_\Gamma A$ is an integral domain by lemma \ref{PolynomLocalFactor}. Since $\Delta \mid \Gamma$, we obtain
	$$\dim C_\Delta R = \dim C_\Gamma A = \mathrm{ht}\:\mathfrak{r}_n - n = \mathrm{ht}\:\mathfrak{p}_n - n,$$
	where the second equality follows from proposition \ref{DimCDeltaPolynom}.
\end{proof}

\begin{rem}\label{remInjective}
 We actually proved in theorem \ref{thmDimCDeltaR} that the map $C_\Gamma A\rightarrow C_{\Delta'}R$ is injective and finite for any $\Delta'\mid \Gamma$. Indeed, the finiteness of this map follows from the finiteness of the map $C_\Gamma A\rightarrow \prod_{\Delta'\mid \Gamma} C_{\Delta'}R$, and since the going-down theorem holds and $C_\Gamma A$ is an integral domain by lemma \ref{PolynomLocalFactor}, we deduce the injectivity.
 \end{rem}
 
 \begin{example}
 In general, remark \ref{remInjective} for finite extensions $A\subset R$ of Noetherian domains with $(0)\in \Gamma$, $(0)\in \Delta$ does not hold. For example, take $A = \Bbbk[x, y]/(y^2 - x^3-x^2)$, $R = \Bbbk[t]$ with $\mathrm{char}\:\Bbbk\neq 2$, and the morphism $A\rightarrow R$ is given by $x\mapsto t^2 - 1$, $y\mapsto t(t^2 - 1)$. Note that this morphism is finite. If $\Gamma = \big((x, y)\big)$ is a flag in $A$, then $C_\Gamma A$ is not an integral domain. There are only two flags in $R$ lying over $\Gamma$, namely, $\Delta_1 = \big((t - 1)\big)$ and $\Delta_2 = \big((t + 1)\big)$. Then $C_{\Delta_i} R$ is an integral domain for $i = 1, 2$. Moreover, $R\otimes_A C_\Gamma A = C_{\Delta_1}R\times C_{\Delta_2}R$ and $C_\Gamma A \hookrightarrow C_{\Delta_1}R\times C_{\Delta_2}R$. However, $C_\Gamma A \rightarrow C_{\Delta_i}R$ is not injective for $i = 1, 2$. The same will hold if we take $\Gamma = \big((0), (x, y)\big)$.
\end{example}

\begin{rem}
	For a general Noetherian ring $R$ theorem \ref{thmDimCDeltaR} does not hold. For example, for any $n \geq 3$ and $0\leq d \leq n - 2$ there exists an excellent regular local ring $R$ of dimension $n$ such that the generic formal fiber, which we consider as the ind-pro completion along the flag consisting of the zero ideal and the maximal ideal, has dimension $d$, see \cite{FJLMPS}.
\end{rem}

\begin{cor}\label{corDimCDeltaR}
    In notation \ref{NotationGeneral}, if $R$ is locally equidimensional, then 
    $$\dim C_\Delta R = \mathrm{ht}\: \mathfrak{p}_n - n.$$
\end{cor}
\begin{proof}
	Since $R$ is catenary and locally equidimensional, we have $\mathrm{ht}(\mathfrak{p}_n/\mathfrak{p}_0) = \mathrm{ht}\: \mathfrak{p}_n - \mathrm{ht}\: \mathfrak{p}_0$, and the assertion follows from theorem \ref{thmDimCDeltaR}.
\end{proof}

\begin{example}
Consider the non-locally equidimensional ring $R = \Bbbk[x, y, z]/(xz, yz)$. Then $\mathrm{Spec}\: R$ is a union of an affine plane and a line in $\A_\Bbbk^3$ intersecting at one point. Consider the flag $\Delta = \big((x, y), (x, y, z)\big)$ in $R$. Geometrically, $\Delta$ is a flag consisting of the line and the intersection point. Thus $\mathrm{ht}\: (x, y) = 0$, $\mathrm{ht}\: (x, y, z) = 2$. Denote $\mathfrak{p} = (x, y)$, $\mathfrak{q} = (x, y, z)$.

Consider the ring $S = C_{(\mathfrak{q})}R = \Bbbk[[x, y, z]]/(xz, yz)$. 
Then $\mathfrak{p}S$ is a prime ideal of $S$ such that $S/\mathfrak{p}S = \Bbbk[[z]]$. Then $S_\mathfrak{p}/\mathfrak{p}S_\mathfrak{p} = \Bbbk((z))$, since the localization by $\mathfrak{p}$ makes $z$ invertible. Since $\mathrm{ht}\:\mathfrak{p} = 0$, we have $\mathfrak{p}R_\mathfrak{p} = 0$, and thus $\mathfrak{p}S_\mathfrak{p} = 0$. Therefore $S_\mathfrak{p} = \Bbbk((z))$, hence $C_{(\mathfrak{p}, \mathfrak{q})}R = \Bbbk((z))$. From this we obtain
$$\dim C_{(\mathfrak{p}, \mathfrak{q})}R = 0 = \mathrm{ht}\: \mathfrak{p} + \mathrm{ht}(\mathfrak{q}/\mathfrak{p}) - 1,$$
since $\mathrm{ht}(\mathfrak{q}/\mathfrak{p}) = 1$, and $\dim C_{(\mathfrak{p}, \mathfrak{q})}R \neq 1 = \mathrm{ht}\:\mathfrak{q} - 1$.
\end{example}

\begin{thm}
     In notation \ref{NotationGeneral}, $C_\Delta R$ is semilocal if and only if $\Delta$ is saturated.
\end{thm}
\begin{proof}
	By remark \ref{remFinGenAlg}, assume that $R$ is a finitely generated $\Bbbk$-algebra. We can assume that $R$ is an integral domain and $\mathfrak{p}_0 = (0)$ by the same reasoning as in the proof of proposition \ref{propSemilocalPolynom}.
	
	As in the proof of theorem \ref{thmDimCDeltaR}, there exists a polynomial subring $A = \Bbbk[x_0, \ldots, x_m]$ in $R$ such that $R$ is finite over $A$ and $\mathfrak{r}_i = \mathfrak{p}_i\cap A = (x_1, \ldots, x_{k_i})$ for some $0 = k_0  < k_1 <\ldots < k_n$, where we assume that $(0) = (x_1, \ldots, x_{k_0})$ for $k_0 = 0$. Moreover, $\mathrm{ht}\:\mathfrak{p}_i = \mathrm{ht}\:\mathfrak{r}_i$. Denote $\Gamma = (\mathfrak{r}_0, \ldots, \mathfrak{r}_n)$.

By remark \ref{remInjective}, the map $C_\Gamma A\rightarrow C_\Delta R$ is finite and injective,  
and $C_\Gamma A$ is an integral domain by lemma \ref{PolynomLocalFactor}. Hence the induced map on spectra  is surjective. Thus $C_\Delta R$ is semilocal if and only if $C_\Gamma A$ is semilocal. The latter is equivalent to the fact that $\Gamma$ is saturated by proposition \ref{propSemilocalPolynom}. Since $\mathrm{ht}\:\mathfrak{p}_i = \mathrm{ht}\:\mathfrak{r}_i$ and $A$ and $R$ are catenary, then $\Gamma$ is saturated if and only if $\Delta$ is saturated, from what the assertion follows.
\end{proof}


\end{document}